\documentclass[reqno]{amsart}
\usepackage[pagewise]{lineno}
\usepackage{mytrees}
\usepackage{booktabs} 
\newtheorem{theorem}{Theorem}[section]
\newtheorem{lemma}[theorem]{Lemma}

\theoremstyle{definition}
\newtheorem{definition}[theorem]{Definition}

\theoremstyle{remark}
\newtheorem{remark}[theorem]{Remark}

\numberwithin{equation}{section}
\usepackage{hyperref}

\newcommand{\mathbbm}[1]{\text{\usefont{U}{bbm}{m}{n}#1}}
\newcommand{\hx}{\hat x}
\newcommand{\hX}{\hat X}
\newcommand{\R}{\mathbb{R}}

\begin{document}

\title{Derivative-free discrete gradient methods}

\author{Håkon Noren Myhr}
\address{Department of Mathematical Sciences, Norwegian University of Science and Technology, Trondheim, Norway}
\email{hakon.noren@ntnu.no}
\thanks{The research was supported by the Research Council of Norway, through the projects DynNoise (No.
339389) and PhysML (No. 338779).}

\author{Sølve Eidnes}
\address{Department of Mathematics and Cybernetics, SINTEF Digital, 0373 Oslo, Norway}
\email{solve.eidnes@sintef.no}

\subjclass[2020]{Primary 65L99, Secondary 65L12}
\date{August 15, 2023}
\keywords{Author Accepted Manuscript version of the paper by Håkon Noren Myhr in Journal of Computational Dynamics, Vol 2024, DOI 10.3934/jcd.2024004}

\begin{abstract}
    Discrete gradient methods are a class of numerical integrators producing solutions with exact preservation of first integrals of ordinary differential equations. In this paper, we apply order theory combined with the symmetrized Itoh--Abe discrete gradient and finite differences to construct an integral-preserving fourth-order method that is derivative-free. The numerical scheme is implicit and a convergence result for Newton's iterations is provided, taking into account how the error due to the finite difference approximations affects the convergence rate. Numerical experiments verify the order and show that the derivative-free method is significantly faster than obtaining derivatives by automatic differentiation. Finally, an experiment using topographic data as the potential function of a Hamiltonian oscillator demonstrates how this method allows the simulation of discrete-time dynamics from a Hamiltonian that is a combination of data and analytical expressions.

\end{abstract}

\maketitle

\section{Introduction}

For an ODE 
\begin{equation}
    \frac{d}{dt} x = f(x), \quad f : \R^n \rightarrow \R^n,
    \label{eq:ODE}
\end{equation}
a first integral is a function $G : \R^n \rightarrow \R$ that remains constant along a solution trajectory $x(t)$ of \eqref{eq:ODE}, i.e.\ $G(x(t)) = G(x(t_0))$ for all $t>t_0$. A first integral could represent conservation laws for physical quantities, such as the energy and angular momentum in the Kepler problem \cite[Ch. I.2]{Hairer06} or the total mass in a chemical reaction \cite[Ch. IV.1]{Hairer06}. A particular example of a first integral is the Hamiltonian $H : \R^{2d} \rightarrow \R$ of a canonical Hamiltonian system, which is a system of the form 
\begin{equation}
    \dot x = S\nabla H(x), \quad \text{where} \; S := \begin{bmatrix}
      0 & I\\
      -I & 0
    \end{bmatrix} \in \R^{2d \times 2d},
    \label{eq:hamiltonian_ODE}
\end{equation}
with $I$ being the $d \times d$ identity matrix.
 That the Hamiltonian is a first integral can be seen by calculating
\begin{equation}
    \frac{d}{dt}H(x) = \nabla H(x)^T S \nabla H(x) = 0.
    \label{eq:energy_preservation}
\end{equation}
This is true for any system $\dot x = S\nabla H(x)$ where $S$ is a skew-symmetric matrix, and not only for canonical Hamiltonian systems. In physics, one could think of $H(x)$ as representing the total energy of the system and thus canonical Hamiltonian systems could be used to represent systems with energy conservation. Hamiltonian systems are also used in the setting of computational statistics when sampling from a probability distribution \cite{nealMCMCUsingHamiltonian2011}. Here, the volume preservation (symplecticity \cite[Ch. VI]{Hairer06}) of the Hamiltonian dynamics simplifies the sampling procedure and allows for more efficient sampling algorithms called Hamiltonian Monte Carlo \cite{nealMCMCUsingHamiltonian2011}.

In this paper, we will study and construct methods that preserve first integrals such as $H(x)$ under discretization in time, while not relying on derivatives of $H$.

\subsection{Discrete gradient methods}
In some cases, first integrals are preserved by Runge--Kutta methods. For instance, any Runge--Kutta method preserves linear first integrals \cite{SHAMPINE19861287} and symplectic Runge--Kutta methods preserve quadratic invariants: $H(x) = x^TAx$ when $A$ is square and symmetric matrix \cite{cooperStabilityRungeKuttaMethods1987a}. However, for $n > 2$, no Runge--Kutta method can preserve all polynomial invariants of degree $n$ \cite{celledoniEnergypreservingRungeKuttaMethods2009}. Discrete gradient methods, on the other hand, are numerical integrators specifically constructed to preserve first integrals of a general form. They were defined formally for the first time in \cite{gonzalezTimeIntegrationDiscrete1996a}. A discrete gradient (DG) of a first integral $H$ is a function $\overline \nabla H : \R^n \times \R^n \rightarrow \R^n$ that satisfies the properties
\begin{align*}
    \overline \nabla H(x,\hx)^T(\hx - x) &=  H(\hx) - H(x),\\
    \overline \nabla H(x,x) &= \nabla H(x).
\end{align*}
Note that, as remarked in \cite{eidnes2022order}, the second property follows from the first if we require the discrete gradient to be differentiable in the second argument. A discrete gradient method (DGM) is a time-stepping scheme for a skew-gradient system, providing approximations $\hx \approx x(t+h)$ given $x \approx x(t)$ by
\begin{equation}
    \hx = x + h S \overline \nabla H(x,\hx).
    \label{eq:general_dgm}
\end{equation}
DGMs preserve invariants since, by the skew-symmetry of $S$, we have that
\begin{align*}
    H(\hx) - H(x) &= \overline \nabla H(x,\hx)^T(\hx - x)\\
    &= h \overline \nabla H(x,\hx)^T  S \overline \nabla H(x,\hx) = 0,
\end{align*}
which could be seen as a discrete version of \eqref{eq:energy_preservation}.

These methods rely on approximations of the gradient $\overline \nabla H \approx \nabla H$, for which there exist several explicitly defined examples. The most common are the midpoint \cite{gonzalezTimeIntegrationDiscrete1996a} and average vector field (AVF) \cite{hartenUpstreamDifferencingGodunovType1983} discrete gradients, which are second-order approximations of the gradient, and the Itoh--Abe discrete gradient \cite{itoh1988hamiltonian}, which is of first order. 

\subsection{Derivative-free numerical algorithms}
Derivative-free numerical algorithms are useful in the context of optimization when seeking to minimize a function even if the gradient is not available. In optimization research, multiple authors have argued that using finite-difference approximations to replace exact derivatives is prohibitively expensive and sensitive to noise. However, the authors of \cite{shiNumericalPerformanceDerivativeFree2021} demonstrate that finite-difference approximations replacing derivatives within for instance the L-BFGS algorithm are competitive to well-known derivative-free optimization algorithms. Furthermore, optimization libraries such as the Python-based SciPy library \cite{2020SciPy-NMeth} apply finite-difference approximations to approximate Jacobian and Hessian matrices when these are not supplied by the user and needed in specific algorithms. Specifically, finite-difference approximations used to approximate the Jacobian in Newton's method are discussed in \cite{kelleyIterativeMethodsLinear1995a,dennis_schnabel_optimization}. In the setting of numerical integration and preservation of first integrals, the Itoh--Abe discrete gradient represents a derivative-free approximation to the gradient of $H(x)$. Indeed, it is observed by \cite{miyatakeEquivalenceSORtypeMethods2018} that the Itoh--Abe discrete gradient method for gradient systems is equivalent to the derivative-free optimization method, coordinate descent, when applied to solve linear systems. Considering the issue of obtaining samples from probability distributions where direct sampling is difficult, Itoh--Abe discrete gradients have been leveraged within the Hamiltonian Monte Carlo scheme to enable a derivative-free approach \cite{mcgregorConservativeHamiltonianMonte2022}.

\section{Derivative-free discrete gradient methods}

The original discrete gradient method \eqref{eq:general_dgm} is restricted to have the same order as the discrete gradient $\overline \nabla H$, but higher-order extensions have also been developed \cite{McLaren04integral, Quispel08new}. A general order theory to construct discrete gradient methods of arbitrary order was presented in \cite{eidnes2022order}. Replacing $S$ in \eqref{eq:general_dgm} with an approximation $\overline S(x,\hx,h)$ allows for the construction of higher-order schemes, even though discrete gradients can at most be second-order approximations of the gradient \cite{mclachlan1999geometric}. The approximation $\overline S(x,\hx, h)$ is required to be skew-symmetric and to satisfy $\overline S(x,x,0) = S(x)$. 

The construction of $\overline S(x,\hx,h)$ requires potentially higher-order derivatives of the first integral $H(x)$ and of the discrete gradient $\overline \nabla H(x,\hx)$, which might be expensive or infeasible to compute. Automatic differentiation is an increasingly common approach to obtain derivatives of functions, especially in machine learning settings where functions may be outputs of neural networks, but this requires the function to be implemented in a programming language that supports automatic differentiation and could be computationally expensive.

In this paper, we propose a fourth-order derivative-free integral-preserving method. This uses the symmetrized Itoh--Abe discrete gradient and an $\overline S(x,\hx, h)$ that relies on finite difference approximations of derivatives of the discrete gradient $\overline \nabla H(x,\hx)$ and the Hessian of $H(x)$. We will first present a general fourth-order discrete gradient method, derived using the order theory of \cite{eidnes2022order}. Then we specify the scheme by choosing to use the symmetrized Itoh--Abe discrete gradient and present appropriate finite difference approximations that ensure that the order is kept while the scheme becomes derivative-free.

\begin{theorem}[Fourth-order DGM]\label{def:general_4th_dgm}
    Consider the Hamiltonian system \eqref{eq:hamiltonian_ODE}. Let $\overline \nabla H$ be a second-order symmetric discrete gradient of $H$,
    and $S_4(x,\hx,h)$ the skew-symmetric matrix
    \begin{equation}\label{eq:s4}
    \begin{aligned}
        S_4(x,\hx,h) = S &+ \frac{h}{2} S \bigg [ Q(x,\frac{x+2\hx}{3})-Q(\hx,\frac{2x + \hx}{3})\bigg ] S\\ &- \frac{h^2}{12} \, S\nabla^2 H(\bar{x})S \nabla^2 H(\bar{x})S,
    \end{aligned}
\end{equation}
    with $\bar x := \frac{x+\hx}{2}$ and
    \begin{equation}
        Q(x,\hx) := \frac{1}{2} (D_2 \overline{\nabla}H (x,\hx)^T - D_2 \overline{\nabla}H(x,\hx)),
    \end{equation}
where $D_2 \overline{\nabla}H(x,\hx)$ is the derivative of the discrete gradient with respect to\ the second argument. Then the discrete gradient method
    \begin{equation}\label{eq:dgm4}
        \hx = x + h S_4(x,\hx,h) \overline{\nabla} H(x,\hx)
    \end{equation}
is of fourth order.
\end{theorem}
\begin{proof}
    Consider
    \begin{equation}\label{eq:dgm3}
        \hx = x + h S_3(x,\hx,h) \overline{\nabla} H(x,\hx)
    \end{equation}
    with
    \begin{equation}\label{eq:s3}
    \begin{aligned}
        S_3(x,\hx,h) = S &+ h S Q(x,\frac{x+2\hx}{3})S - \frac{h^2}{12} \, S\nabla^2 H(\bar{x})S \nabla^2 H(\bar{x})S.
    \end{aligned}
    \end{equation}
    We have that \eqref{eq:dgm3} is of third order if \eqref{eq:s3} satisfies the necessary order conditions given in \cite[Table 5]{eidnes2022order}.
    To that end, we note that \eqref{eq:s3} can be written as 
    \begin{align*}
        S_3(x,\hx,h) = & \, \frac{1}{2}(1+1) S + h \frac{1}{2} \Big(S Q(x,\frac{x+2\hx}{3})S + S Q(x,\frac{x+2\hx}{3})S\Big)\\
        & \, + h^2 (-\frac{1}{24}) \Big(S\nabla^2 H(\bar{x})S \nabla^2 H(\bar{x})S + S\nabla^2 H(\bar{x})S \nabla^2 H(\bar{x})S\Big),
    \end{align*}
    which is of the form \cite[Eq. 70]{eidnes2022order}, with
    \begin{align*}
        \sum_j b_{\ab j} &= b_{\ab 1} =  \frac{1}{2},\\
        \sum_j b_{\tabb j} &= b_{\tabb 1} = \frac{1}{2},\\
        \sum_{j} b_{\tabb j} \phi_{\tabb j1}(\ab) &= b_{\tabb 1} \phi_{\tabb 11}(\ab) = \frac{1}{3},\\
        \sum_j b_{\aaabbb j} &= b_{\aaabbb 1} = -\frac{1}{24},
    \end{align*}
    where we use that
    \begin{equation*}
        \frac{x + 2\hx}{3} = x + \frac{2}{3} h S \nabla H(x) +   \mathcal{O}(h^2)
    \end{equation*}
    is a G-series, as defined in \cite{eidnes2022order}, with $\phi(\ab) = \frac{2}{3}$.
    Furthermore, we can ignore the order conditions for $\ttabbb$ and $\atabbb - \taabbb$ since, as noted in \cite{eidnes2022order}, these correspond to terms involving $Q(x,x)$, which is zero for any symmetric discrete gradient. Then the necessary conditions for \eqref{eq:dgm3} being of at least order 3 are satisfied.

    By \cite[Theorem II.3.2]{Hairer06}, the adjoint method of \eqref{eq:dgm3}, i.e.
    \begin{align*}
        \hx &= x + h S_3(\hx,x,-h) \overline{\nabla} H(x, \hx),
    \end{align*}
    is also of third order. Any affine combination of $S$'s that gives a discrete gradient method of order $p$ is another $S$ that gives a discrete gradient method of at least order $p$. Thus
    \begin{align*}
    S_4(x,\hx,h) = \frac{1}{2}(S_3(x,\hx,h) + S_3(\hx,x,-h))
    \end{align*}
    gives a scheme of at least order $3$. Since it gives a symmetric method, we have by \cite[Theorem II.3.2]{Hairer06} that this is of order 4.
\end{proof}

\begin{remark}
    It should be noted that discrete gradient methods can also be applied to dissipative systems, in which case the skew-symmetric matrix $S$ in \eqref{eq:hamiltonian_ODE} is replaced by a negative semi-definite matrix, and preserve the dissipation property \cite{mclachlan1999geometric}. However, in this case the higher-order schemes derived using the theory of \cite{eidnes2022order} will not lead to dissipative discrete gradient methods in general. For instance, $S_4$ given by \eqref{eq:s4} is not guaranteed to be negative semi-definite if $S$ is.
\end{remark}

If the discrete gradient is the AVF discrete gradient, then $Q(x,\hx) = 0$ for any $x, \hx$ \cite{eidnes2022order}, and \eqref{eq:dgm4} becomes the fourth-order method of Quispel and McLaren \cite{Quispel08new}.
We aim at constructing a derivative-free, fourth-order DGM. This could for instance be achieved by starting with the first-order Itoh--Abe DG, which is defined below:

\definition[Itoh--Abe discrete gradient]\label{def:itoh_abe_method}
We use the following short-hand notation for the various vector concatenations or shifts along dimension $m$:
\begin{equation}    
\begin{aligned}
    \hat X_m &:= [\hx_1,\hx_2,\dots,\hx_m,x_{m+1},\dots,x_n]^T,\\
    X_m &:= [x_1,x_2,\dots,x_m,\hx_{m+1},\dots,\hx_n]^T.
    \label{eq:concat_m}
\end{aligned}
\end{equation}
Then the Itoh--Abe discrete gradient \cite{itoh1988hamiltonian} is defined for a Hamiltonian $H : \R^n \rightarrow \R$, by 
\begin{equation}
    \overline \nabla_{\text{IA}} H(x, \hx) := \begin{bmatrix}
        \frac{H(\hx_1,x_2,x_3,\dots,x_n) - H(x_1,x_2,x_3,\dots,x_n)}{\hx_1 - x_1}\\ 
        \frac{H(\hx_1,\hx_2,x_3,\dots,x_n) - H(\hx_1,x_2,x_3,\dots,x_n)}{\hx_2 - x_2}\\
        \vdots\\
        \frac{H(\hx_1,\hx_2,\dots,\hx_{n-1},\hx_n) - H(\hx_1,\hx_2,\dots,\hx_{n-1},x_n)}{\hx_n - x_n}\\
    \end{bmatrix} = 
    \begin{bmatrix}
        \frac{H(\hX_1) - H(\hX_0)}{h_1}\\ 
        \frac{H(\hX_2) - H(\hX_1)}{h_2}\\
        \vdots\\
        \frac{H(\hX_n) - H(\hX_{n-1})}{h_n}\\
    \end{bmatrix},
    \label{eq:itoh_abe_def}
\end{equation}
assuming $x_i\neq \hx_i$ and letting $h_i = \hx_i - x_i$. In the event that $x_i = \hx_i$ we define the $i$-th element of the discrete gradient to be given by $(\overline \nabla_{\text{IA}} H(x, \hx))_i := \frac{\partial}{\partial x_i} H(\hat{X}_i)$.

This is a first-order approximation of the gradient. However, by computing the average of the Itoh--Abe discrete gradient and its adjoint (switching $x$ and $\hx$), we obtain a symmetric discrete gradient of second order, due to \cite[Th. II.3.2]{Hairer06}. We call this the symmetrized Itoh--Abe (SIA) discrete gradient. It is defined by

\begin{equation}
    \overline \nabla_{\text{SIA}} H(x, \hx) := \frac{1}{2}\bigg (  \overline \nabla_{\text{IA}} H(x, \hx) + \overline \nabla_{\text{IA}} H(\hx, x) \bigg ) \!=\!
    \begin{bmatrix}
        \frac{H(\hX_1) - H(\hX_0) - H(X_0) + H(X_1) }{h_1}\\ 
        \frac{H(\hX_2) - H(\hX_1) - H(X_1) + H(X_2) }{h_2}\\
        \vdots\\
        \frac{H(\hX_n) - H(\hX_{n-1}) - H(X_{n-1}) + H(X_n) }{h_n}
    \end{bmatrix}.
    \label{eq:sia_method}
\end{equation}

These discrete gradients are derivative-free, but only of first and second order; they will thus give first- and second-order integrators if used in the scheme \eqref{eq:general_dgm} with $\overline S(x,\hx,h) =S$. We will now show how to construct a fourth-order derivative-free discrete gradient method, by combining the second-order Itoh--Abe discrete gradient with the fourth-order discrete gradient method in Theorem \ref{def:general_4th_dgm}. Since the fourth-order DGM requires the Hessian of $H$, $\nabla^2H$, and the derivative of the discrete gradient with respect to the second argument, $D_2\overline \nabla H(x,\hx)$, these quantities must be approximated in a derivative-free manner.

\section{Finite difference approximations}

\subsection{Hessian approximation}

Let $e_i$ denote the $i$-th vector in the canonical basis, so that e.g.\ $e_2 = [0,1,0,\dots,0]^T$, and let $\mathbbm{1} := [1,\dots,1]\in\R^n$. We define the following difference operators:
\begin{align*}
    [\Delta_1(H)(x)]_i &:= H(x + \tau e_i) +  H(x - \tau e_i), \\
    [\Delta_2(H)(x)]_{ij} &:= H\big(x + \tau(e_i + e_j)\big) + H\big(x - \tau(e_i + e_j)\big).
\end{align*}
By computing the average of the forward and backward difference approximations, we obtain the following second-order approximation of the Hessian:
\begin{align}
    \nabla^2 H(x) &= \frac{1}{2\tau^2} \bigg(2 \cdot \mathbbm{1}\mathbbm{1}^TH(x) + \Delta_{2}(H)(x) - \Delta_1(H)(x)\mathbbm{1}^T - \mathbbm{1}\Delta_1(H)(x)^T \bigg ) + \mathcal O(\tau^2) \nonumber \\
    &=\nabla^2_{\tau}H(x) + \mathcal O(\tau^2),
\label{hess_fdm}
\end{align}
where $\nabla^2_{\tau}H(x)$ denotes the Hessian approximation. The order is easily verified by Taylor expansion, and the cost of this approximation in terms of evaluations of $H(x)$ is $n^2 + 3n + 1$. If $n>2$, this is less than directly applying the midpoint finite difference rule, which would require $2n^2 + 1$ function evaluations.

\subsection{Derivative of the SIA discrete gradient}\label{ch:derivative_of_SIA} The symmetrized Itoh--Abe discrete gradient is given element-wise by 
\begin{equation}
[ \overline \nabla_{\text{SIA}} H(x, \hx) ]_i = \frac{1}{h_i}\bigg (H(\hat X_i) - H(X_i) - H(\hat X_{i+1}) + H(X_{i+1})    \bigg).
\end{equation}
By considering which variables depend on $\hx_j$ following the definitions of $X_m$ and $\hX_m$ in \eqref{eq:concat_m} we obtain the following expression for the derivative of the SIA discrete gradient:
\begin{equation}
    \begin{aligned}
        \frac{\partial}{\partial \hx_j}[    \overline \nabla_{\text{SIA}} H(x, \hx) ]_i &= 
        \begin{cases} \frac{1}{h_i} \frac{\partial}{\partial \hx_j}\bigg (H(\hat X_{i+1}) - H(\hat X_i)    \bigg) \quad &\text{if } i>j,\\
    \frac{1}{h_i} \frac{\partial}{\partial \hx_j}\bigg ( H( X_i) - H( X_{i+1})    \bigg) \quad &\text{if } i<j,\\
    \frac{1}{h_i} \frac{\partial}{\partial \hx_i}\bigg ( H( X_i) + H(\hat X_{i+1})    \bigg) - \frac{1}{h_i}[    \overline \nabla_{\text{SIA}} H(x, \hx) ]_i
        \quad &\text{if } i=j.
        \end{cases}
    \end{aligned}
    \end{equation}
where the equation when $i=j$ follows from the product rule. Approximating derivatives of $H(x)$ with the midpoint rule we obtain
\begin{equation}
    \frac{1}{h_i} \frac{\partial}{\partial \hx_j}  H( X_i) = \frac{1}{2h_i\tau}\bigg( H( X_i + e_j\tau) - H( X_i - e_j\tau) \bigg) + \mathcal O(\frac{\tau^2}{h}),
\end{equation}
and similar for $X_{i+1}$, $\hX_i$ and $\hX_{i+1}$. Let us denote the total midpoint approximation of the SIA discrete gradient by $D_2^{\tau} \overline \nabla_{\text{SIA}} H(x, \hx)$ such that 
\begin{equation}
    D_2 \overline \nabla_{\text{SIA}} H(x, \hx) =  D_2^{\tau} \overline \nabla_{\text{SIA}} H(x, \hx) + \mathcal O(\frac{\tau^2}{h}).
    \label{midpoint_sia_fdm}
\end{equation}

Computing $D_2^{\tau} \overline \nabla_{\text{SIA}}$ requires $4(n^2 + n)$ evaluations of $H(x)$ which is reduced to $4(n^2 - n)$ when the diagonal is not needed, which is the case when computing $Q_{\tau}(x,\hx)$ due to its skew-symmetry. The derivative approximation of the first order Itoh--Abe method $D_2^{\tau} \overline \nabla_{\text{IA}}$ requires $2(n^2 + n)$ evaluations.
 
\subsection{Finite precision in difference schemes}\label{ch:round_off}
Due to limited machine precision, error from numerical methods or noise, it is often the case that a function $f : \R \rightarrow \R$ only can be evaluated to a finite precision. Let us denote this function by $f_{\epsilon}(x) := f(x) + \epsilon(x)$ for some bounded error  $\epsilon(x) \leq \overline \epsilon$. Computing a finite difference approximation needs to take the finite precision into account when choosing the discretization step size $\tau$, as discussed for instance in \cite{kelleyIterativeMethodsLinear1995a,shiNumericalPerformanceDerivativeFree2021}. For the midpoint difference scheme, we have that

\begin{align}
    \frac{\partial}{\partial x} f_{\epsilon}(x) &= \frac{1}{2\tau}\bigg( f(x + \tau) - f(x - \tau) + \epsilon(x + \tau) - \epsilon(x - \tau) \bigg) + \mathcal O(\tau^2) \nonumber \\
    &= \frac{1}{2\tau}\bigg( f(x + \tau) - f(x - \tau)  \bigg) + \mathcal O\bigg(\tau^2 + \frac{\overline \epsilon}{\tau} \bigg ).
    \label{first_finite_pres}
\end{align}
For $\tau > 0$, the leading order term $\tau^2 + \frac{\overline \epsilon}{\tau}$ is convex and thus minimized by $\tau = \overline \epsilon^{\frac{1}{3}}$. For the Hessian approximation, we use a second-order approximation for the second derivative and find that
\begin{equation}
\begin{aligned}
    \frac{\partial^2}{\partial x^2} f_{\epsilon}(x) =& \frac{1}{2\tau^2}\bigg( f(x) -  2f(x + \tau) - 2f(x - \tau) +  f(x + 2\tau) +  f(x - 2\tau) \bigg) \\ &+  \mathcal O\bigg(\tau^2 + \frac{\overline \epsilon}{\tau^2} \bigg ).
    \label{second_finite_pres}
\end{aligned}
\end{equation}
By a similar argument as above, the leading order term is minimized by $\tau = \overline \epsilon^{\frac{1}{4}}$.

\subsection{Derivative-free fourth-order DGM}

\begin{lemma}\label{lem:err_in_S}
    Let $S_4^{\tau}$ be given by the same form as $S_4$ in Definition \ref{def:general_4th_dgm} with the following substitutions:
    \begin{align*}
        D_2 \overline \nabla_{\text{SIA}} H(x, \hx) &\rightarrow D_2^{\tau_1} \overline \nabla_{\text{SIA}} H(x, \hx) \quad \text{in $Q(x,\hx)$},\\
        \nabla^2 H(\bar x) &\rightarrow  \nabla^2_{\tau_2}H(\bar x).
    \end{align*}
    Here $\tau_1$ and $\tau_2$ are the finite difference step sizes for the derivative of the SIA DG approximated by \eqref{midpoint_sia_fdm} and the Hessian approximated by \eqref{hess_fdm}, respectively.  Assume that the Hamiltonian can only be evaluated to a finite precision $H(x) + \epsilon(x)$, where the error $\epsilon(x)$ is bounded by $\overline \epsilon$. Then 
    \begin{equation}
        S_4^{\tau}(x,\hx,h) = S_4(x,\hx,h) +\mathcal O\bigg (\tau_1^2 + \frac{ \overline \epsilon}{\tau_1} + h^2(\tau_2^2 + \frac{ \overline \epsilon }{\tau_2^2}  )  \bigg ).
    \end{equation}
\end{lemma}
\begin{proof}
    Let 
    \begin{align*}
        Q_{\tau_1}(x,\hx) := \frac{1}{2} (D_2^{\tau_1} \overline{\nabla}H (x,\hx)^T - D_2^{\tau_1} \overline{\nabla}H(x,\hx)).
    \end{align*}
    Then
    \begin{align*}
        Q_{\tau_1}(x,\hx) = Q(x,\hx) + \mathcal O\bigg(\frac{1}{h}(\tau_1^2 + \frac{ \overline \epsilon}{\tau_1})\bigg),
    \end{align*}
    due to the finite difference error in \eqref{midpoint_sia_fdm} and the finite precision error in \eqref{first_finite_pres}. Similarly, we find  
    \begin{align*}
S\nabla^2_{\tau_2} H(\bar{x})S \nabla^2_{\tau_2} H(\bar{x})S &= S\nabla^2 H(\bar{x})S \nabla^2 H(\bar{x})S + \mathcal O\bigg( \tau_2^2 + \frac{ \overline \epsilon}{\tau_2^2} \bigg),
\end{align*}
    due to \eqref{hess_fdm} and \eqref{second_finite_pres}.  In total, we arrive at
\begin{align*}
    S_4^{\tau}(x,\hx,h) = S&+ \frac{h}{2} S \bigg [ Q_{\tau_1}(x,\frac{x+2\hx}{3})-Q_{\tau_1}(\hx,\frac{2x + \hx}{3})\bigg ] S \\ &- \frac{h^2}{12} \, S\nabla^2_{\tau_2} H(\bar{x})S \nabla^2_{\tau_2} H(\bar{x})S\\
    = S&_4(x,\hx,h) + \mathcal O\bigg (\tau_1^2 + \frac{ \overline \epsilon}{\tau_1} + h^2(\tau_2^2 + \frac{ \overline \epsilon }{\tau_2^2}  )  \bigg ).
\end{align*}
\end{proof}

\begin{remark}
    Selecting $\tau_1$ and $\tau_2$ to minimize the leading order terms, following the discussion above, we find that
    \begin{equation}
    S_4^{\tau}(x,\hx,h) = S_4(x,\hx,h) +\mathcal O\bigg (\overline \epsilon^{\frac{2}{3}} + h^2\overline \epsilon^{\frac{1}{2}} \bigg ).
    \label{eq:S_df_error}
    \end{equation}
\end{remark}
Since $\overline \nabla_{\text{SIA}}H$ is a symmetric discrete gradient and does not require derivatives, the combination with the gradient-free skew-symmetric $S_4^{\tau}$ gives a fourth-order derivative-free method by
\begin{equation}\label{df_dgm_4}
    x_{n+1} = x_{n} + hS_4^{\tau}(x_n,x_{n+1},h)\overline \nabla_{\text{SIA}}H(x_n,x_{n+1}).
\end{equation}
This yields
\begin{equation*}
\begin{aligned}
    x_{n+1} &= x_{n} + hS_4(x_n,x_{n+1},h)\overline \nabla_{\text{SIA}}H(x_n,x_{n+1})  + \mathcal O\bigg (h\overline \epsilon^{\frac{2}{3}} + h^3\overline \epsilon^{\frac{1}{2}} \bigg ) \\
    &= x(t_n + h) +  \mathcal O\bigg (h\overline \epsilon^{\frac{2}{3}} + h^3\overline \epsilon^{\frac{1}{2}} + h^4 \bigg ),
\end{aligned}
\end{equation*}
assuming that $\tau_1$ and $\tau_2$ are chosen optimally regarding error from the finite precision evaluation of $H(x)$. Since $x_{n+1}$ is defined implicitly, Newton's method will be utilized to enable the solution of initial value problems.

\section{Convergence of Newton's method}

Newton's method could be utilized to find the root $x^*\in\R^n$ such that $F(x^*) = 0$ of a differentiable function $F: \R^n \rightarrow \R^n$. Let $F'(x):= \frac{\partial}{\partial x} F(x)$ be the Jacobian, which we assume is invertible for all $x \in D \subset \R^n$. The iterations of Newton's method are then given by 
\begin{equation}
    x^{(i+1)} = x^{(i)} + F'(x^{(i)})^{-1}F(x^{(i)}).
\end{equation}
As proved in Theorem 5.1.1 in \cite{kelleyIterativeMethodsLinear1995a}, assuming that $F(x) = 0$ has a solution $x^*$ and that $F'(x)$ is Lipschitz continuous for all $x\in D$, the iterations exhibit quadratic convergence; i.e., for some constant $\tilde C>0$ we have that
\begin{equation*}
    \|x^{(i+1)} - x^*\| \leq \tilde C \|x^{(i)} - x^*\|^2.
\end{equation*}

Consider first using the derivative-dependent scheme \eqref{eq:dgm4} with the SIA DG to solve initial value problems on Hamiltonian form. That is, assuming $x := x(t_n)$ we aim to find $\hx \approx x(t_{n+1})$ such that $F(\hx) = 0$, where 
\begin{equation}
\begin{aligned}
    F(\hx) :=& \; \hx - x - h S_4(x,\hx,h) \overline \nabla_{\text{SIA}} H(x,\hx),\\
    \text{and} \quad F'(\hx) =& \; I - h \frac{\partial}{\partial \hx} S_4(x,\hx,h) \overline \nabla_{\text{SIA}} H(x,\hx) - h S_4(x,\hx,h)  \frac{\partial}{\partial \hx} \overline \nabla_{\text{SIA}} H(x,\hx).
    \label{eq:exact_sia_4_newton}
\end{aligned}
\end{equation}
Considering then the gradient-free method of \eqref{df_dgm_4}, we replace $S_4$ by $S_4^\tau$. Also ignoring the second term of $F'$, we arrive at approximations of $F$ and $F'$ given by
\begin{equation}
\begin{aligned}
    F_{\tau}(\hx) :=& \; \hx - x - h S^{\tau}_4(x,\hx,h) \overline \nabla_{\text{SIA}} H(x,\hx),\\
    \quad F_{\tau}'(\hx) :=& \; I - h S^{\tau}_4(x,\hx,h)  D_2^{\tau}\overline \nabla_{\text{SIA}} H(x,\hx),
\end{aligned}
\label{eq:df_dgm_4_newton}
\end{equation}
where $S^{\tau}_4(x,\hx,h)$ is defined in Lemma \ref{lem:err_in_S} and $ D_2^{\tau}\overline \nabla_{\text{SIA}} H(x,\hx)$ is the approximation found in Equation \eqref{midpoint_sia_fdm}. This begs the question of how the error introduced by $F_{\tau}$ and $F'_{\tau}$ affects the convergence of Newton's method, compared to the solution obtained by \eqref{eq:exact_sia_4_newton}. Theorem 5.4.1 in \cite{kelleyIterativeMethodsLinear1995a} presents a general convergence proof for Newton iterations under inexact evaluation of $F$ and $F'$. The following theorem considers the special case in which the inexact iterations are given by the approximations in \eqref{eq:df_dgm_4_newton}.
\begin{theorem}\label{thm:df_dgm_4_newton}
Let $x^*$ satisfy $F(x^*) = 0$, where $F$ is defined as in \eqref{eq:exact_sia_4_newton}. Consider the inexact Newton iterations
    \begin{equation}
    x^{(i+1)} = x^{(i)} + F_{\tau}'(x^{(i)})^{-1}F_{\tau}(x^{(i)}),
    \label{eq:newton_iter}
\end{equation}
using $F_{\tau}$ and $F_{\tau}'$ given by \eqref{eq:df_dgm_4_newton}. Assume that the Hamiltonian can be evaluated to finite precision $H(x) + \epsilon(x)$, with $\epsilon(x) \leq \overline \epsilon$. Let the finite difference step size for the derivative $D_2^{\tau_1}$ be given by  $\tau_1 = \overline \epsilon^{\frac{1}{3}}$ and for the Hessian $\nabla^2_{\tau_2}$ be $\tau_2 = \overline \epsilon^{\frac{1}{4}}$.  Then
\begin{equation}
    \|x^{(i+1)} - x^*\| \leq C \bigg ( \|x^{(i)} - x^*\|^2 + \big ( \overline \epsilon^{\frac{2}{3}}+ h^2 +  h^3\overline \epsilon^{\frac{1}{2}} \big )\|x^{(i)} - x^*\| +   h\overline \epsilon^{\frac{2}{3}} + h^3\overline \epsilon^{\frac{1}{2}}  \bigg ).
\end{equation}
\end{theorem}

\begin{proof}
    For inexact iterations by
    \begin{align*}
    \tilde F(x) &= F(x) + \epsilon(x) \\
    \tilde F'(x) &= F'(x) + \Delta(x)
\end{align*}
Theorem 5.4.1 in \cite{kelleyIterativeMethodsLinear1995a} states that the convergence of the iterations are determined by
\begin{equation*}
    \|x^{(i+1)} - x^*\| \leq C \bigg ( \|x^{(i)} - x^*\|^2 + \|\Delta(x^{(i)}) \| \|x^{(i)} - x^*\| +  \|\epsilon(x^{(i)})\| \bigg ).
\end{equation*}
In this specific case the exact  $F(\hx)$ is given by \eqref{eq:exact_sia_4_newton}, and from Lemma \ref{lem:err_in_S} and \eqref{df_dgm_4} it is clear that
\begin{equation}
     F_{\tau}(\hx) - F(\hx) =  \mathcal O\bigg (h\overline \epsilon^{\frac{2}{3}} + h^3\overline \epsilon^{\frac{1}{2}} \bigg ),
\end{equation}
meaning it is possible to find a constant $C_1 \geq 0$ such that
\begin{equation*}
     \|F_{\tau}(\hx) - F(\hx)\| \leq C_1  (h\overline \epsilon^{\frac{2}{3}} + h^3\overline \epsilon^{\frac{1}{2}}  ).
\end{equation*}
To simplify the notation, let us omit the arguments for $S_4$ and $H$ and write $D_2$ to represent $D_2\overline \nabla_{\text{SIA}}H$ and similarly $D_2^{\tau}$ for the finite difference approximation. For the Jacobian, we observe that
\begin{align*}
    \| F_{\tau}'(\hx) - F'(\hx) \| \leq& \; h \big \| S^{\tau}_4  D_2^{\tau} -S_4  D_2 \big \| + h \| \frac{\partial}{\partial \hx} S_4  \| \\
    =& \; h \big \| S^{\tau}_4  D_2^{\tau} -S_4  D_2 \big \| + \mathcal O(h^2)\\
    \leq& \;  h\| (S^{\tau}_4 - S_4 )D_2 \big \| + h\| S_4^{\tau}(D_2^{\tau} - D_2 ) \big \| + \mathcal O(h^2)\\
    =& \; \mathcal O( h\overline \epsilon^{\frac{2}{3}} + h^3\overline \epsilon^{\frac{1}{2}} + \overline \epsilon^{\frac{2}{3}} + h^2) \\
    =& \; \mathcal O(\overline \epsilon^{\frac{2}{3}}+ h^2 +  h^3\overline \epsilon^{\frac{1}{2}}  ).
\end{align*}
The first equality follows from
\begin{equation*}
     \frac{\partial}{\partial \hx}S_4 = \frac{h}{2}S \frac{\partial}{\partial \hx} \big(Q(x,\frac{1}{3}x+\frac{2}{3}\hx)-Q(\hx,\frac{2}{3}x+\frac{1}{3}\hx)\big) S + \mathcal O(h^2).
\end{equation*}
To derive the second inequality, we have used that for matrices $A,B$ and $\hat A,\hat B$,
\begin{align*}
    \hat A \hat B -  A  B &= \hat A \hat B - \hat AB + \hat A B  -  A  B  \\
    &= \hat A(\hat B- B) + (\hat A- A) B.
\end{align*}
The leading order terms of $S_4^{\tau} - S_4$ are due to \eqref{eq:S_df_error}, while those for $D_2^{\tau} - D_2$ are due to \eqref{midpoint_sia_fdm} and \eqref{first_finite_pres}.

\end{proof}

\subsection{Computational complexity}

The computational complexity of the different Itoh--Abe methods could be estimated by considering the number of evaluations of the Hamiltonian needed to perform one time-step. We make the simplifying assumption that the same number of iterations are needed for the three methods to reach convergence for Newton's method. A more thorough analysis would consider the constant terms governing the convergence, i.e.\ studying $\|F'(x)\|$ for the different methods and how this could lead to varying numbers of iterations. However, in practice, for sufficiently small step sizes ($h < 0.1$), approximately the same number of iterations are required to converge. We thus consider the number of evaluations required for one inexact Newton iteration \eqref{eq:newton_iter}.

We note that the first and second order methods leave $S^{\tau} = S$ constant, while $S^{\tau}_4$ requires one Hessian evaluation and two evaluations of $D_2^{\tau}\overline \nabla H$, where the diagonal could be excluded due to skew-symmetry. In summary, the number of evaluations of $H(x)$ needed to perform one Newton iteration is found in Table \ref{tab:eval_number}.

 \begin{table}[hb]

 \begin{center}
 \begin{tabular}{l l | l  l l | l}
 \toprule
    Method &Order&$ \overline \nabla H$ &  $D_2^{\tau} \overline \nabla H$ &   $S^{\tau}$ & Total \\ 
 \midrule
     IA &$1$& $2n$  & $2(n^2 + n)$ & $0$ & $2n^2 +4n$  \\
     SIA &$2$& $4n$ & $4(n^2 + n)$  & $0$ & $4n^2 + 8n$  \\
     SIA $4$ & $4$ & $4n$ &  $4(n^2 + n)$ & $9n^2 -5n + 1$ &  $13n^2 + 3n + 1$  \\
 \bottomrule
 \end{tabular}
\vskip 0.1in
 \caption{Numbers of evaluations of  $H(x)$ for the different derivative-free methods.}
 \label{tab:eval_number}
 \end{center}
 \end{table}

\section{Numerical experiments}

\subsection{Convergence of $S_4^{\tau}$, $F_{\tau}$ and $F'_{\tau}$}

\begin{figure}[!htb]
    \centering
    \includegraphics[width=0.49\textwidth]{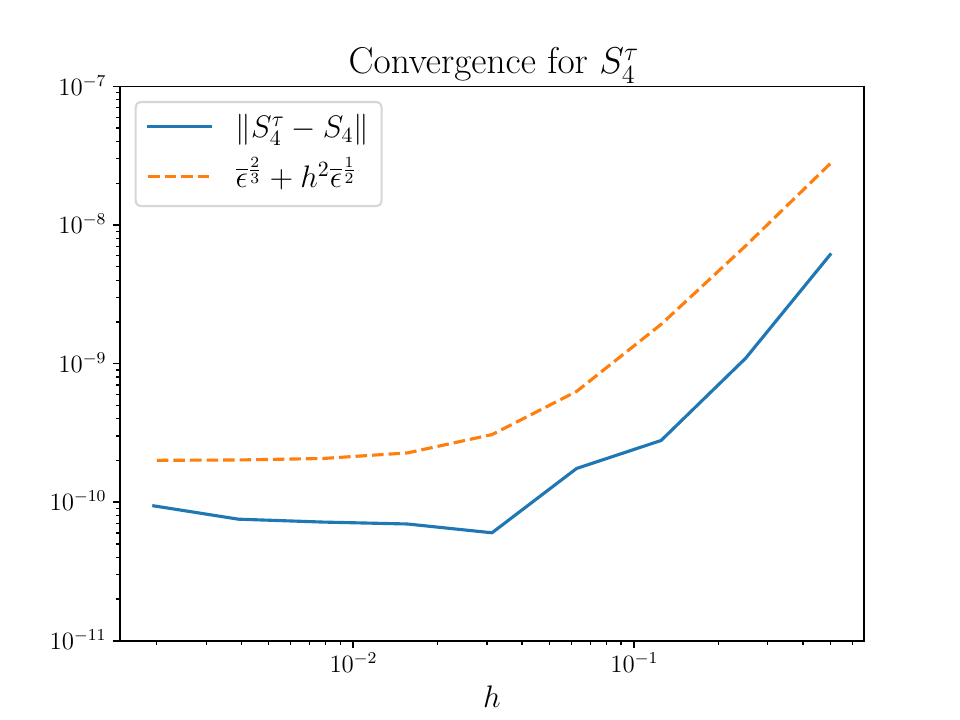}
    \includegraphics[width=0.49\textwidth]{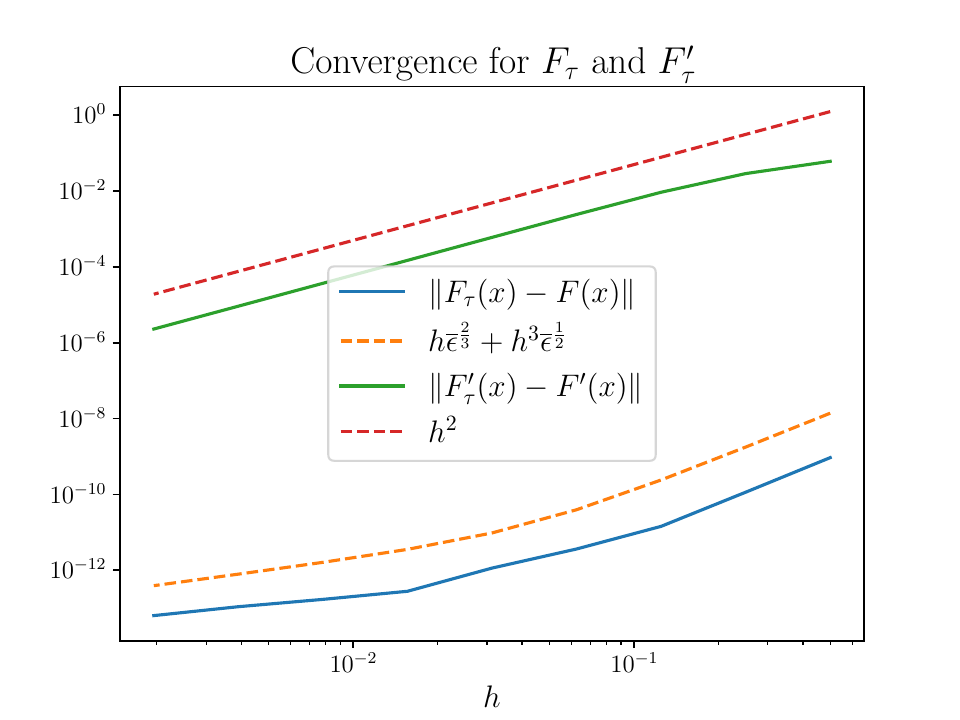}
    \caption{
        Numerical and theoretical convergence of $S_{\tau}$, $F_{\tau}$ and $F'_{\tau}$ when the step size in time $h$ decreases. Dotted lines represent the theoretical convergence rates (Lemma \ref{lem:err_in_S} and Theorem \ref{thm:df_dgm_4_newton}) and the solid lines the numerical results.
    }
    \label{fig:convergence_S_F}
\end{figure}

First, we investigate numerically the error of $S_4^{\tau}, F_{\tau}$ and $F'_{\tau}$ to test the theoretical convergence results in Lemma \ref{lem:err_in_S} and Theorem \ref{thm:df_dgm_4_newton}. This is achieved by comparing the finite difference approximations $S_4^{\tau}, F_{\tau}$ and $F'_{\tau}$ to $S_4, F$ and $F'$, which are computed using automatic differentiation provided by the Python package Autograd. This numerical experiment is performed on the double pendulum Hamiltonian presented in the next section and the results are shown in Figure \ref{fig:convergence_S_F}. By using double precision floating point numbers, we have that $\overline \epsilon \approx 10^{-15}$ and thus set the finite difference step-sizes to be $\tau_1 = 10^{-5}$ and $\tau_2 = 10^{-4}$. By considering the leading order terms in the errors of $S_4^{\tau}, F_{\tau}$ and $F'_{\tau}$ plotted with dotted lines in Figure \ref{fig:convergence_S_F} we see that the errors are in agreement with the theoretical results in Lemma \ref{lem:err_in_S} and Theorem \ref{thm:df_dgm_4_newton}.

\subsection{Convergence of numerical solution}

\begin{figure}[htb]
    \centering
    \includegraphics[trim = 0 130 0 0, clip, width=0.49\textwidth]{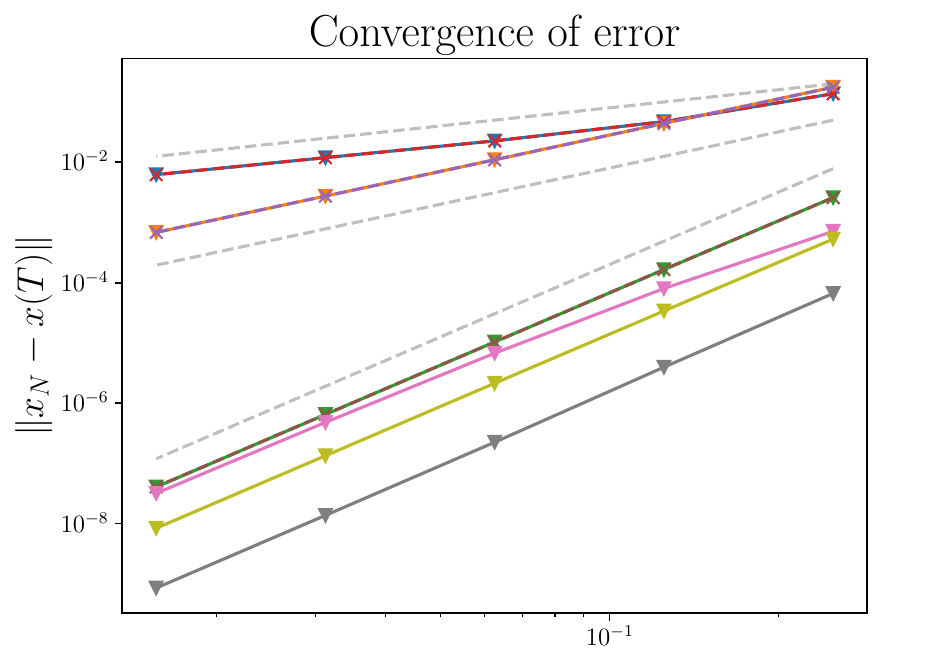}
    \includegraphics[trim = 0 130 0 0, clip,width=0.49\textwidth]{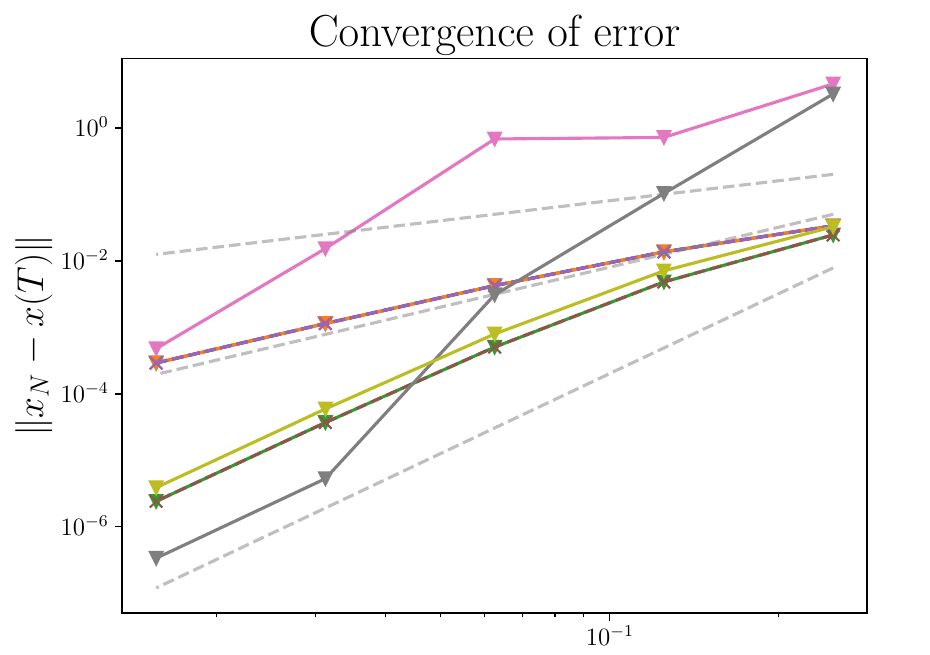}
    \includegraphics[trim = 0 130 0 0, clip,width=0.49\textwidth]{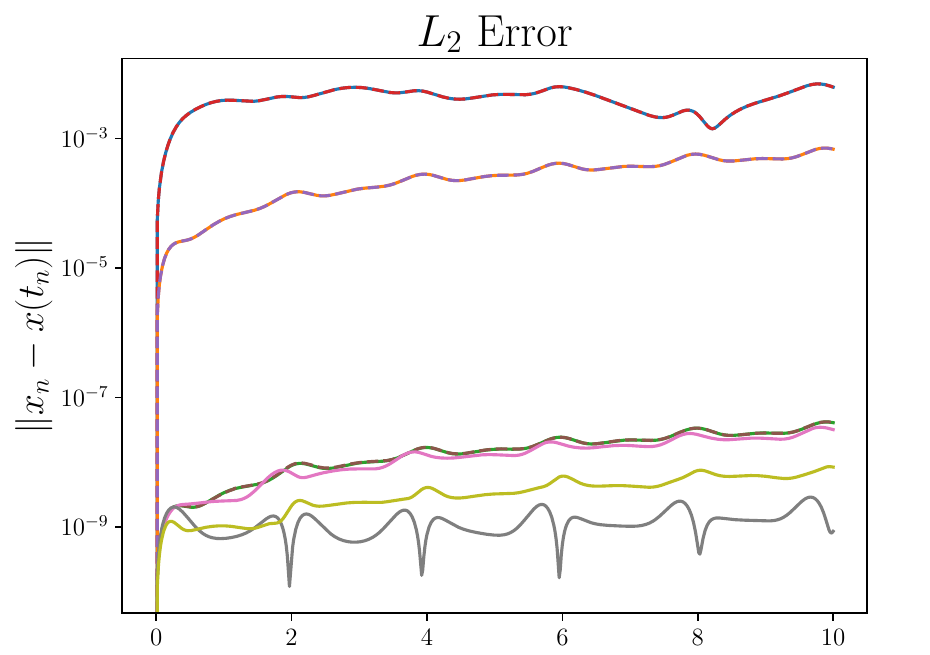}
    \includegraphics[trim = 0 130 0 0, clip,width=0.49\textwidth]{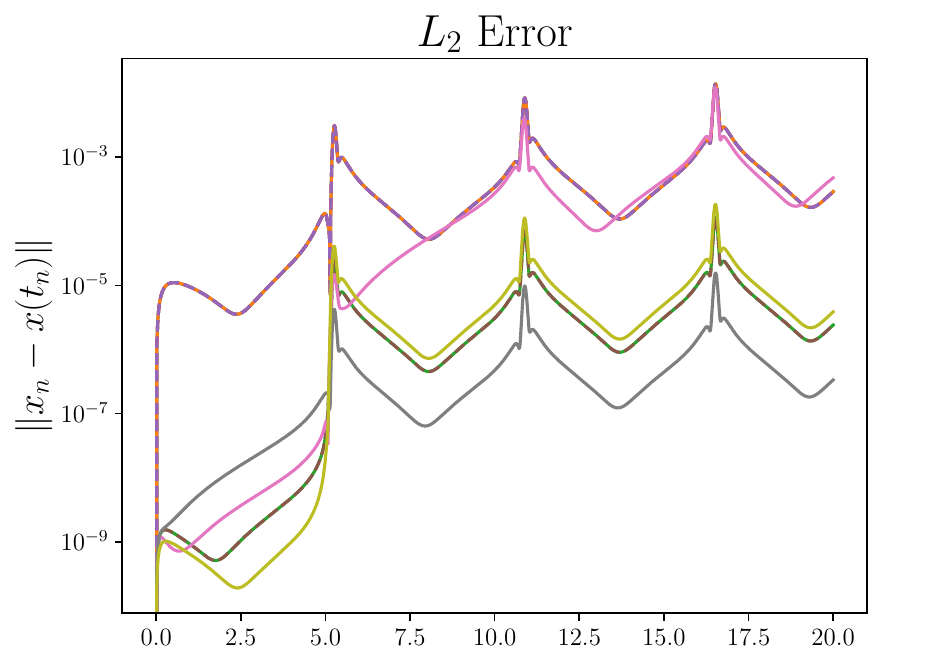}
    \includegraphics[trim = 5 0 5 345, clip,width=0.5\textwidth]{figures/convergence_ljo.pdf}
    \caption{
        Error results for the double pendulum (left column) and the Lennard--Jones oscillator (right column). The top row shows the global error at end time, $\|x_N - x(T)\|_2$, plotted against the step size $h$. The bottom row shows the $L_2$ error over time, $\|x_n - x(t_n)\|_2$, where the smallest step size was used. Dashed lines display the derivative-free (DF) methods. The gray dotted lines represent order $h^1$, $h^2$ and $h^4$ respectively.
    }
    \label{fig:error_convergence_double_pendulum}
\end{figure}

We now consider the convergence of the numerical solution of the double pendulum problem and the Lennard--Jones oscillator. The double pendulum system has a Hamiltonian that is not separable, where  $x=[q_1,q_2,p_1,p_2]^T \in \R^4$ and the Hamiltonian is given by
\begin{equation*}
H(q_1,q_2,p_1,p_2) = \frac{\frac{1}{2}p_1^2 + p_2^2 - p_1p_2\cos(q_1-q_2)}{1+\sin^2(q_1-q_2)} - 2\cos(q_1) -\cos(q_2).
\end{equation*}
Here, $q_i$ and $p_i$ denote the angle and angular momentum of pendulum $i = 1,2$. The Lennard--Jones oscillator has a potential function $U(q)$ that is used to model repulsion and attraction in molecular dynamics \cite{johnsonLennardJonesEquationState1993}. The oscillator on Hamiltonian form exhibits dynamics with abrupt changes, making energy preservation in discrete time challenging \cite{leimkuhlerSimulatingHamiltonianDynamics2005}. The Hamiltonian is given by 
\begin{equation*}
    H(q_1,p_1) = \frac{1}{2}p_1^2 + U(q_1) = \frac{1}{2}p_1^2 + \frac{1}{4}\bigg(\frac{1}{q_1^{12}} - \frac{2}{q_1^6}\bigg).
\end{equation*}
Here, $q_1$ and $p_1$ denote the position and momentum of the oscillator. The initial conditions for the double pendulum and the Lennard--Jones oscillator are given by $x_0 = [0.1,0.2,0.25,-0.3]^T$ and $x_0 = [1.21,0.34]^T$, respectively.

\begin{figure}[!htb]
    \centering
    \includegraphics[width=0.49\textwidth]{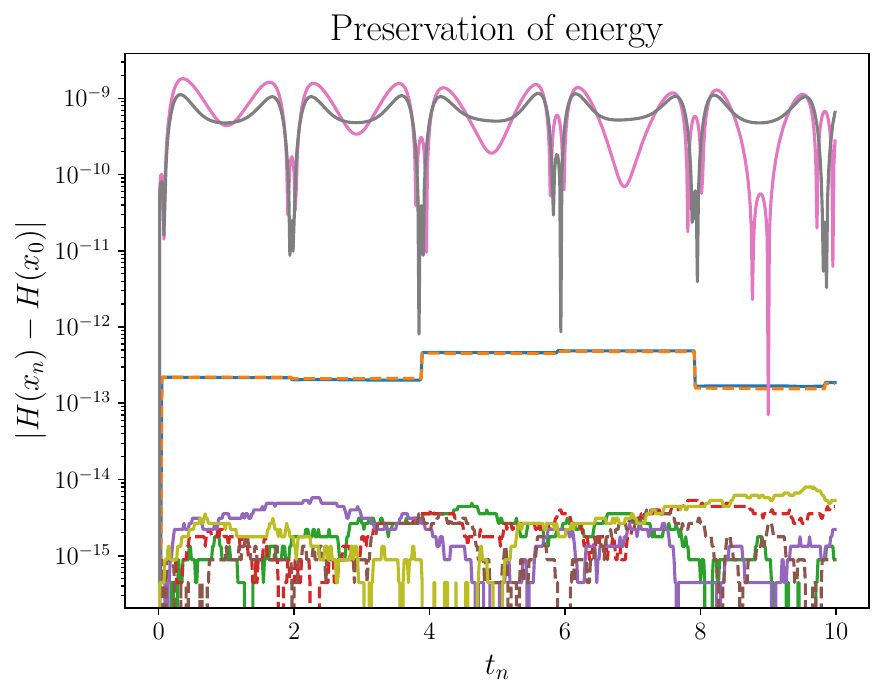}
    \includegraphics[width=0.49\textwidth]{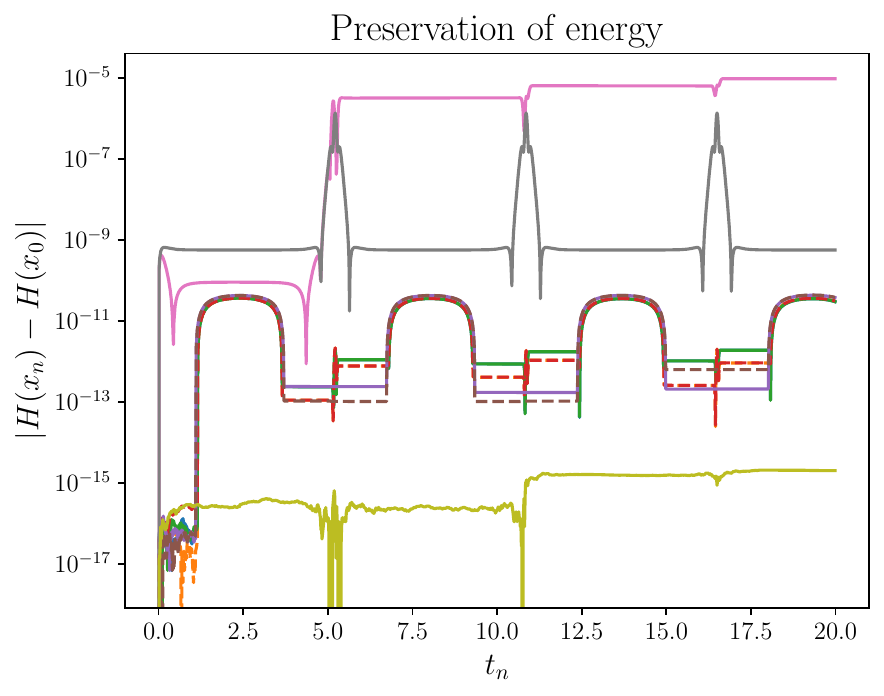}
    \includegraphics[trim = 5 0 5 345, clip,width=0.5\textwidth]{figures/l2error_double_pendulum.pdf}
    \caption{
        Energy error for the double pendulum (left) and the Lennard--Jones oscillator (right) plotted over time where the smallest step size was used. Dashed lines represent the derivative-free (DF) methods.
    }
    \label{fig:energy_error_double_pendulum}
\end{figure}

We compare the derivative-free DGM in \eqref{df_dgm_4}, denoted as SIA $4$ DF, against the similar method (SIA $4$) where the only difference is that the Hessian and the derivative of the discrete gradient is computed using automatic differentiation. Similarly, we compare derivative-free first-order Itoh--Abe method of Definition \ref{def:itoh_abe_method} (IA DF) and the symmetrized method (SIA DF) against implementations using automatic differentiation denoted as IA and SIA in the figures. Both methods require no modifications of the skew-symmetric matrix $S$ and are solved with Newton iterations using 
\begin{equation*}
    \begin{aligned}
        F_{\tau}(\hx) :=& \; \hx - x - h S \overline \nabla_{\text{DG}} H(x,\hx)\\
        \quad F_{\tau}'(\hx) :=& \; I - h S  D_2^{\tau}\overline \nabla_{\text{DG}} H(x,\hx),
    \end{aligned}
    \end{equation*}
where $\text{DG} \in \{\text{IA},\text{SIA} \}$ and $D_2^{\tau}\overline \nabla_{\text{DG}} H(x,\hx)$ is the midpoint finite difference approximation of the derivative of the discrete gradient with respect to\ its second argument, as derived in Section \ref{ch:derivative_of_SIA} for the SIA method. The same methodology is applied when deriving the finite difference scheme for the IA discrete gradient. Additionally, we compare to the classic $4$-stage Runge--Kutta method (RK4) \cite[p. 138]{hairerSolvingOrdinaryDifferential2000}, the fourth-order DGM method presented in \cite{eidnes2022order} (DGM$4$) in addition to the fourth order, three stage, mono-implicit Runge--Kutta method (MIRK$4$) found in \cite{muirOptimalDiscreteContinuous1999}.

The convergence order is investigated numerically by keeping the end time $T$ fixed, decreasing the step size $h = \frac{T}{N}$ and computing the global error $\|x_N - x(T)\|_2$ for the various methods, where $x_N$ is the numerical approximation of the solution at time $t_N=T$. The results are presented in Figure \ref{fig:error_convergence_double_pendulum}, together with the $L_2$ error over time, $\|x_n - x(t_n)\|_2$, and the energy error $|H(x_0) - H(x_n)|$ is shown in Figure \ref{fig:energy_error_double_pendulum}. Furthermore, the number of evaluations of $H(x)$ are presented in Figure \ref{fig:computational_time_double_pendulum} and computational time is compared to accuracy in the work-precision diagram found in Figure \ref{fig:work_precision}. The convergence results are in line with what to expect from the theoretical discussion and it is clear that the derivative-free methods have the same precision as the same methods based on automatic differentiation. However, considering the computational time of the different methods used, the derivative-free methods are significantly faster, which is made clear from studying the work-precision diagram in Figure \ref{fig:work_precision}. In particular, we observe that the derivative-free methods are approximately $50$ times faster than the methods relying on exact derivatives. However, the methods based on finite differences require a significantly higher number of evaluations of the Hamiltonian $H(x)$. Hence, if the Hamiltonian was more computationally costly to evaluate, these results may look different. 

\begin{figure}[!htb]
    \centering
    \includegraphics[width=0.49\textwidth]{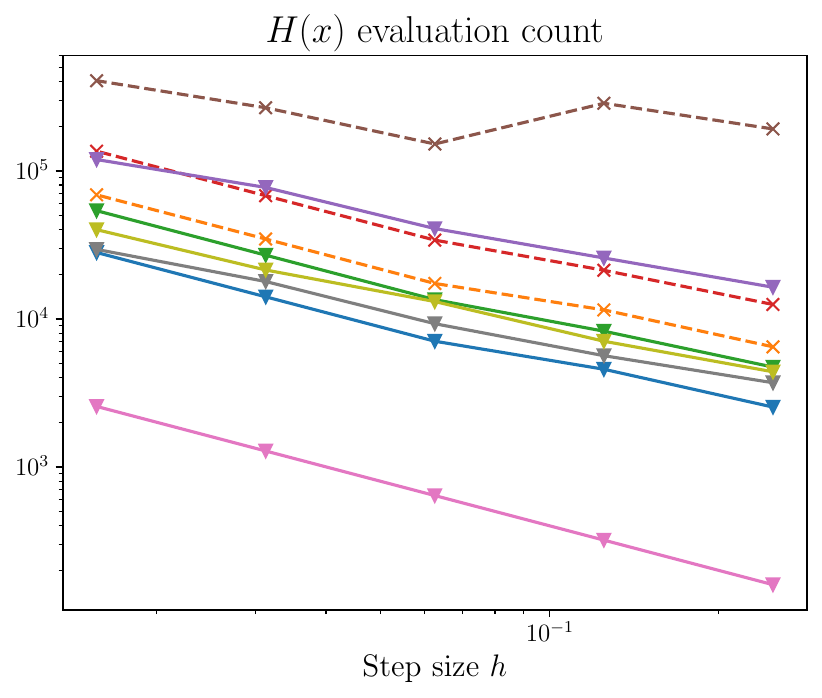}
    \includegraphics[width=0.49\textwidth]{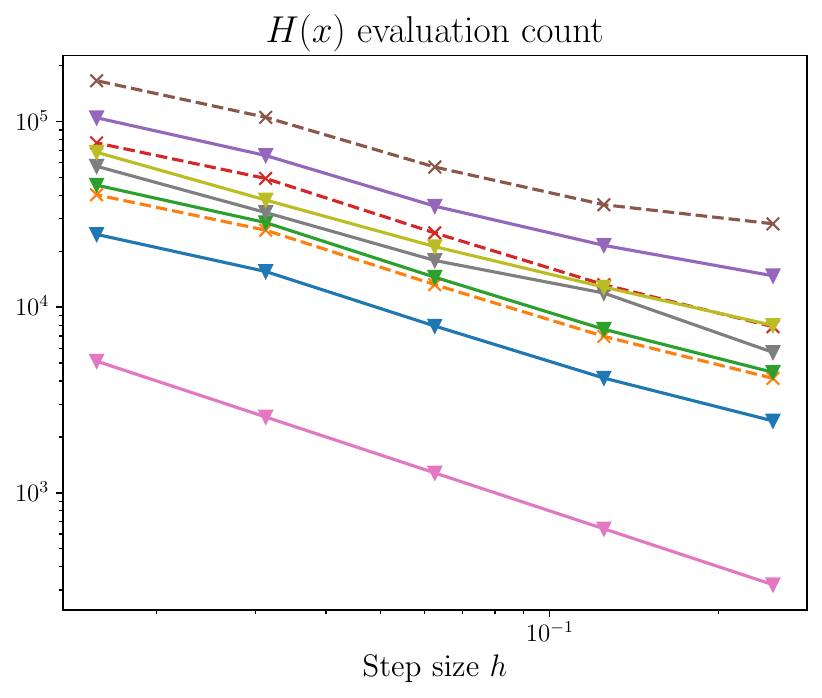}
    \includegraphics[trim = 5 0 5 345, clip,width=0.5\textwidth]{figures/convergence_double_pendulum.pdf}
    \caption{
        Numbers of evaluations of $H(x)$  for the double pendulum (left) and the Lennard--Jones oscillator (right) over multiple experiments with different $h$.
    }
    \label{fig:computational_time_double_pendulum}
\end{figure}

\begin{figure}[!htb]
    \centering
    \includegraphics[width=0.49\textwidth]{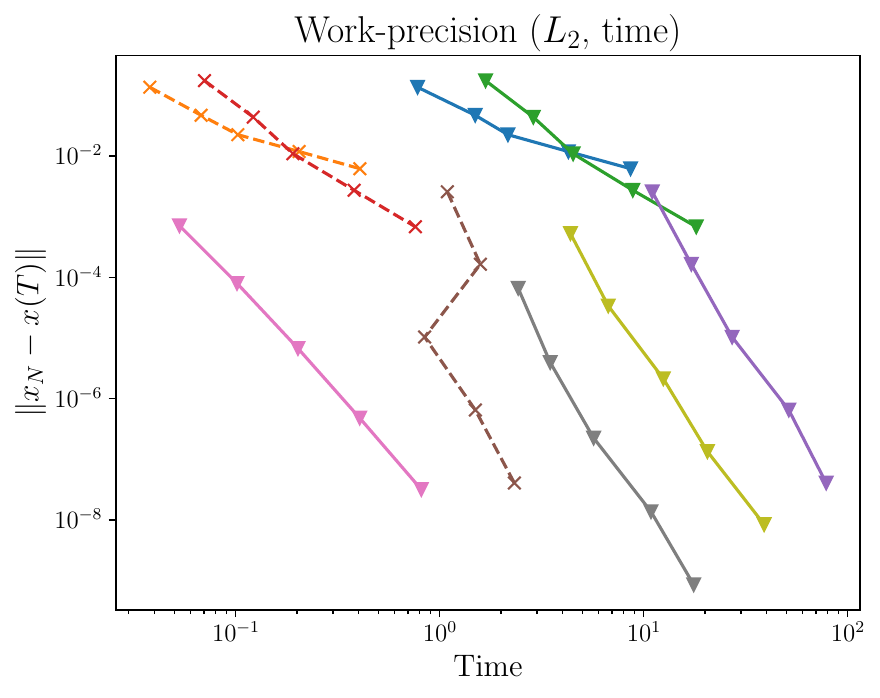}
    \includegraphics[width=0.49\textwidth]{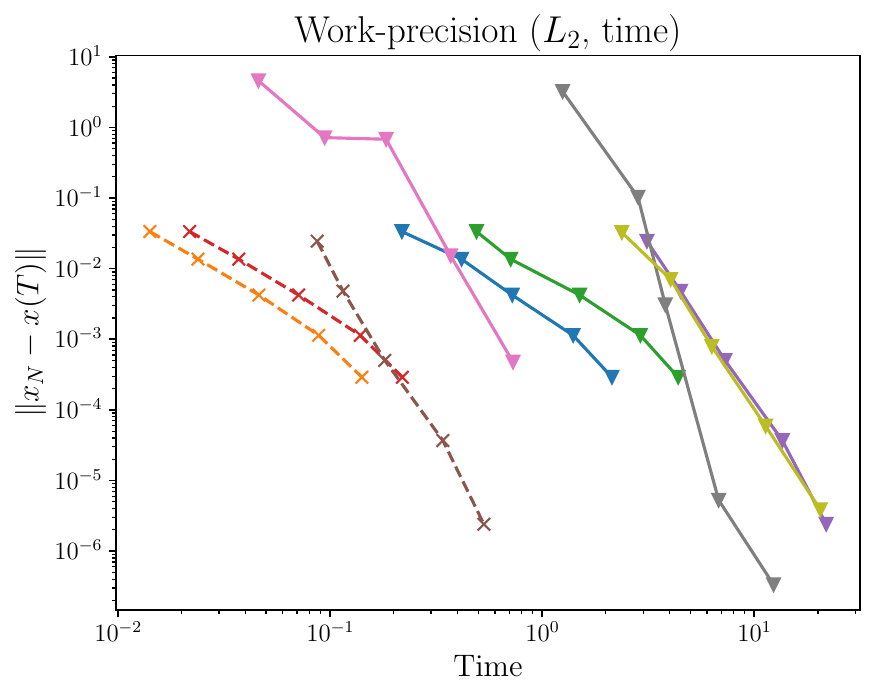}
    \includegraphics[width=0.49\textwidth]{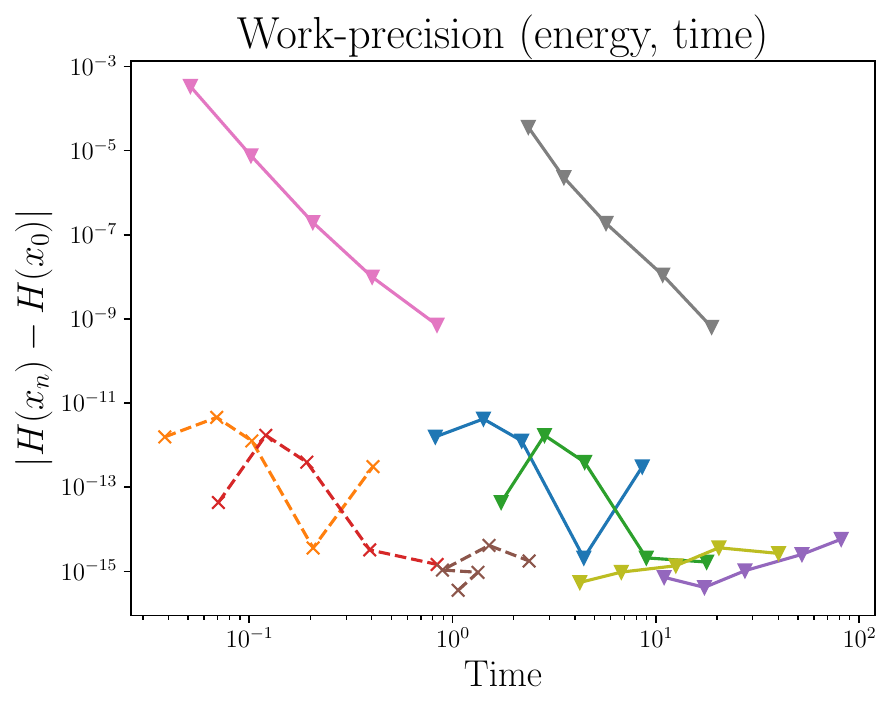}
    \includegraphics[width=0.49\textwidth]{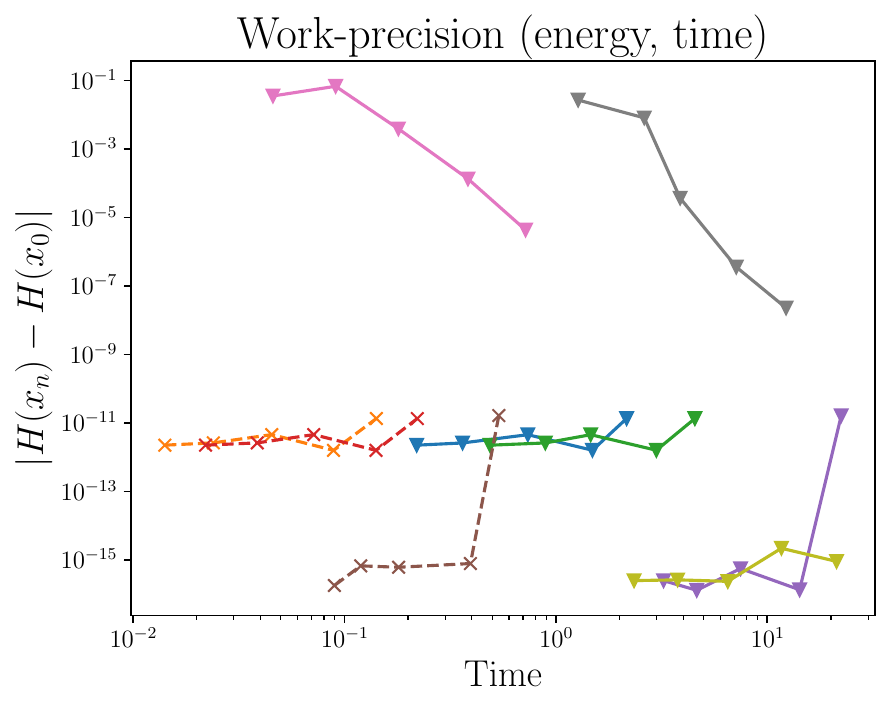}
    \includegraphics[trim = 5 0 5 345, clip,width=0.5\textwidth]{figures/convergence_double_pendulum.pdf}
    \caption{
        Work-precision diagrams where the $L_2$ (top row) and energy error (bottom row) is plotted against computational time for the double pendulum (left column) and the Lennard--Jones oscillator (right column) over multiple experiments with different step size in time $h$. Dashed lines display the derivative-free (DF) methods.
    }
    \label{fig:work_precision}
\end{figure}

\subsection{Implementation details}

For the discrete gradients, we have that $\overline \nabla H(x,x) = \nabla H(x)$. Hence, for constructing an algorithm that is gradient-free, we need to initialize the Newton iterations with $x_{n+1}^{(0)} \neq x_n$. Integrating a trajectory $x_n \approx x(t_n)$ for $n = 0,\dots,N$, we initialize the iterations with 

\begin{align*}
 x_{n+1}^{(0)} =& \begin{cases}
         x_n + h\delta \quad &\text{if} \; n = 0\\
        x_n + x_{n} - x_{n-1} \quad &\text{else}.\\
    \end{cases}
\end{align*}
 Here, $\delta$ is a sample from a standard normal distribution $\delta \sim \mathcal N(0, I)$. For $n\neq 0$, using $ x_{n+1}^{(0)}  = x_n + x_{n} - x_{n-1}$ is equivalent with assuming that the vector field $f(x(t))$ is constant for $t \in [t_{n-1},t_{n}]$.

 The Newton iterations are terminated when $\|F(x_{n+1}^{(i)})\|\leq \text{TOL} = 10^{-11}$ or the number of iterations satisfies $i>20$. The code for the implementation of the numerical experiments can be found on GitHub. \footnote{\url{https://github.com/hakonnoren/derivative_free_discrete_gradients}}

 \subsection{Topographic Hamiltonian} \label{ch:topographic_hamiltonian}

 \begin{figure}[!htb]
    \centering
    \includegraphics[width=0.49\textwidth]{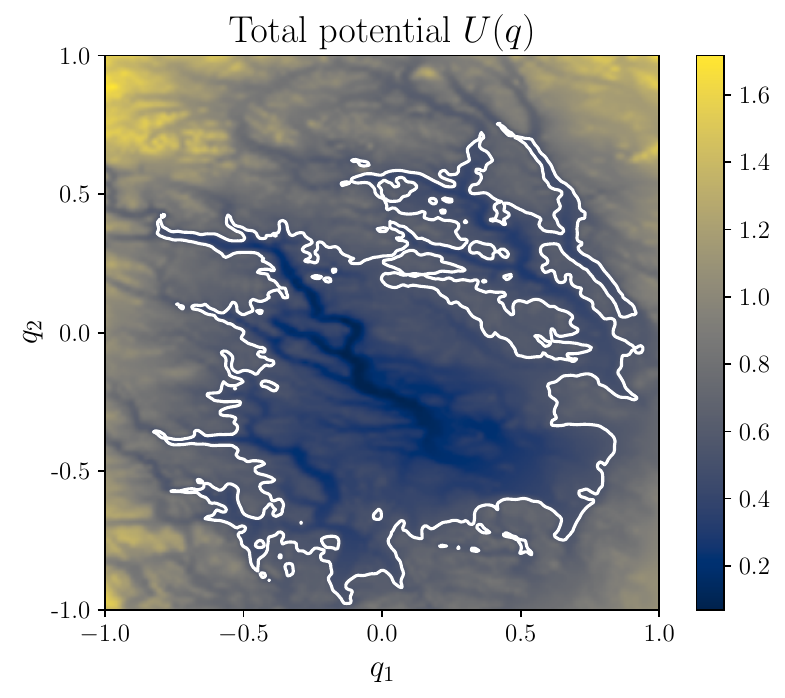}
    \includegraphics[width=0.49\textwidth]{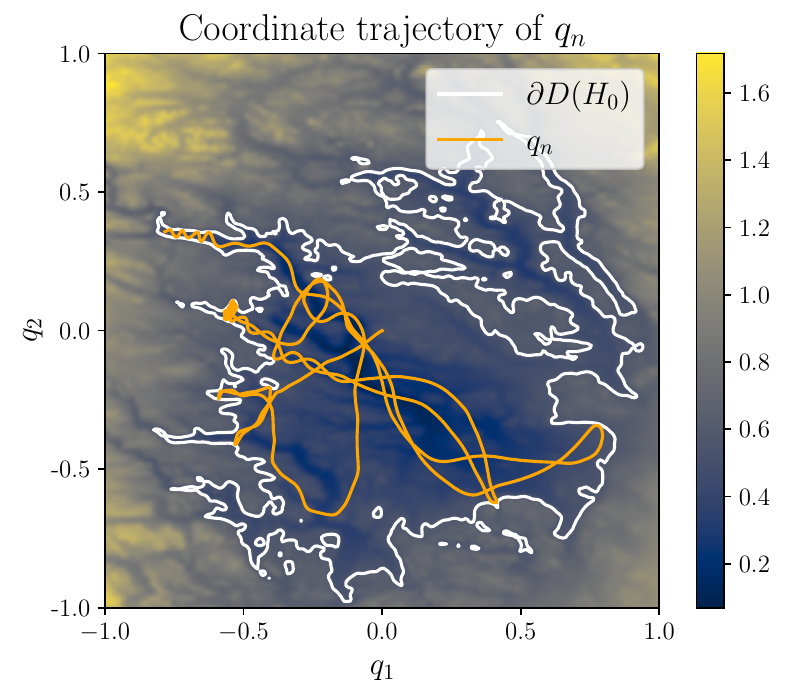}
    \includegraphics[width=0.49\textwidth]{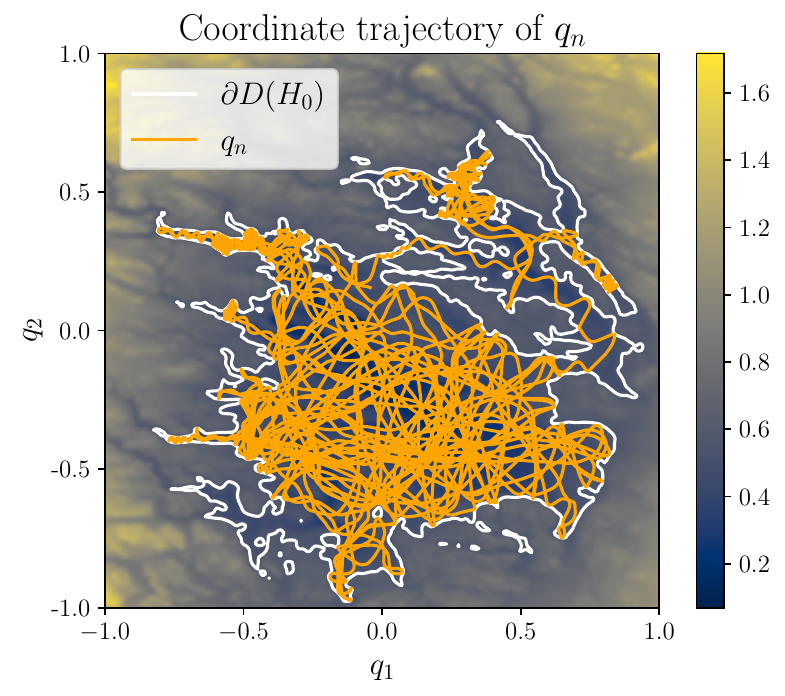}
    \includegraphics[width=0.49\textwidth]{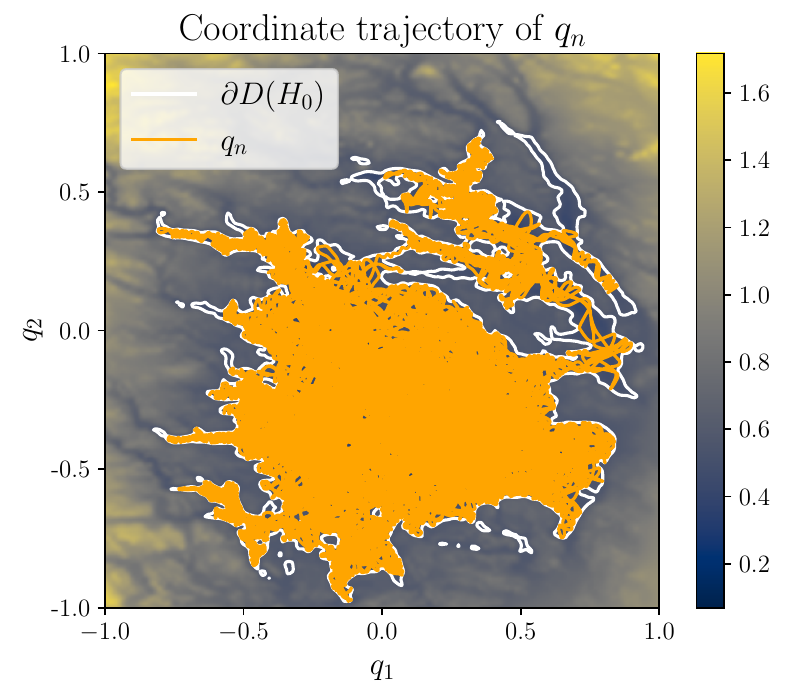}
    \caption{
    The total potential $U(q)$ (upper left) and the trajectory of the topographic Hamiltonian $q_n \approx q(t_n)$ in discrete time for $t_n \leq t_N$ with $t_N = 20,200,1000$ from the upper to the lower right quadrant. $\partial D(H_0)$ is the level set of the initial energy $H(q_0,p_0)$. The color in the heatmap corresponds to the value of the total potential $U(q)$.
    }
    \label{fig:topo_trajectory}
\end{figure}
To demonstrate how the derivative-free DGM allows for the numerical integration of Hamiltonians where exact differentiation might be more involved or inaccessible, we will consider the following Hamiltonian with a potential that is the interpolation of topographic (altitude) data. Let $q \in D := [-1,1] \times [-1,1] \subset \R^2$ be coordinates and $p\in\R^{2}$ the generalized momenta. Then we let the topographic Hamiltonian be defined as
\begin{equation}
    H(q,p) = U_{\text{top}}(q) + \frac{1}{2}(q^Tq + p^Tp),
\end{equation}
where $U_{\text{top}} : D \rightarrow [0,1]$ maps coordinates $q \in D$ to a normalized and interpolated altitude value $[0,1]$. Note that the total potential is given by 
\begin{equation*}
    U(q) = U_{\text{top}}(q) + \frac{1}{2}q^Tq.
\end{equation*}
Let now $q(t) \in D$ be a coordinate trajectory of the Hamiltonian system with initial condition $q(t_0) = q_0$ and $p(t_0) = p_0$. Since the Hamiltonian is conserved along the trajectory $q(t),p(t)$ the total potential energy $U(q)$ is bounded by $H(q_0,p_0)$, meaning that the coordinate trajectory $q(t)$ is confined to
\begin{equation}
    D(H_0) := \{q \in \R^2 \; \big | \; U(q)  \leq H(q_0,p_0)\}.
\end{equation}
Assuming $q_0,p_0$ is chosen such that $D(H_0) \subset D$, $q(t)$ will always remain within the domain of the interpolant $U_{\text{top}}$. When discretizing the Hamiltonian dynamics in time, a discrete gradient method will, by conserving the energy, ensure the stability of the solutions in the sense that $q_n \in D(H_0) \subset D$ for all $n$.

\begin{figure}[!htb]
    \centering
    \includegraphics[width=0.65\textwidth]{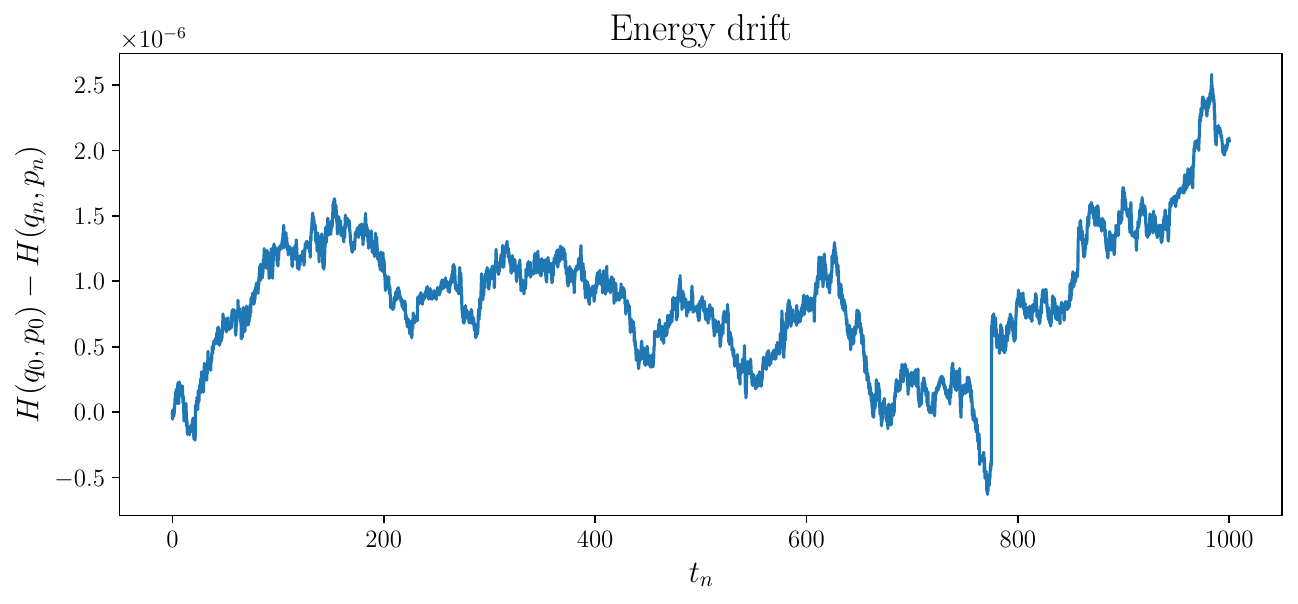}
    \caption{
        Energy drift of the topographic Hamiltonian in discrete time.
    }
    \label{fig:topo_energy_drift}
\end{figure}

By using PyGMT \cite{uieda_leonardo_2023_7772533}, a Python interface for the Generic Mapping Tools \cite{wessel2019generic}, topographic data of an area approximately $222$ km by $222$ km centered around Ilsetra in Norway approximately at degrees $(61.221^{\circ}, 10.525^{\circ})$ is obtained pointwise in a matrix $Q \in \R^{122 \times 122}$. The potential $U_{\text{top}}$ is constructed by interpolating $Q$ with a cubic spline that is normalized such that $U_{\text{top}}(q) \in [0,1]$ for all $q \in D$. The interpolation is done using the SciPy library \cite{2020SciPy-NMeth}. 

A numerical experiment is demonstrated by integrating the Hamiltonian system $N = 50\;000$ steps with $h = 0.02$ starting at $[q_0,p_0]^T = [0,0,-0.1,0.2]^T$. The initial energy is $H(q_0,p_0) = 0.5598$ and the derivative-free symmetrized Itoh--Abe method (SIA DF) is used and solved with a tolerance of $\text{TOL} = 10^{-7}$. In Figure \ref{fig:topo_trajectory} the trajectory $q_n$ is plotted on a colormap of the total potential $U(q)$ in addition to the boundary of $D(H_0)$ denoted as $\partial D(H_0)$. Figure \ref{fig:topo_energy_drift} shows the energy drift over time. 

From Figure \ref{fig:topo_trajectory} we observe that the discrete trajectory $q_n$ is confined to $D(H_0)$, and the energy drift in Figure \ref{fig:topo_energy_drift} is roughly of the same order as the tolerance of the non-linear solver. While it is feasible to compute the gradient of a cubic spline such as $U_{\text{top}}(q)$, this would in practice require a separate implementation either of the gradient directly, or implementing the interpolation algorithm with a library supporting automatic differentiation. In this regard, the derivative-free methods could be seen as non-intrusive, meaning numerical integration is enabled without larger changes to existing numerical algorithms that are involved when computing the Hamiltonian.

\section{Conclusion}

The derivative-free discrete gradient methods presented in this paper are realized by leveraging the symmetrized Itoh--Abe method and the order theory for discrete gradients in \cite{eidnes2022order}, in addition to finite differences that replace the derivatives involved in $S_4(x,\hx,h)$ and the Jacobian $F'(\hx)$ needed for iterations of Newton's method. Theorem \ref{thm:df_dgm_4_newton} shows that the inaccuracy introduced due to the finite difference approximations is negligible when the finite precision evaluation of the Hamiltonian is only due to floating point precision (where $\overline \epsilon \approx 10^{-15}$). This is confirmed in the numerical experiments where SIA $4$ DF is comparable in accuracy to SIA $4$, where the derivatives are computed using automatic differentiation. However, the derivative-free methods are around $50$ times faster in the experiments presented here. Additionally, we have demonstrated how the derivative-free method allows for the numerical integration of Hamiltonians where differentiation might be more involved. Among the potential future applications of our method is the training of Hamiltonian neural networks \cite{greydanus2019hamiltonian}, where the gradient of the approximated Hamiltonian needs to be computed in every training step. Additionally, the methods presented could be leveraged when sampling from distributions where gradients are expensive or inaccessible, extending the Conservative Hamiltonian Monte Carlo algorithm \cite{mcgregorConservativeHamiltonianMonte2022}.

\bibliographystyle{amsplain}
\bibliography{ref}
\end{document}